\documentclass{article}

\usepackage{hyperref}

\usepackage{amsthm,amssymb}
\usepackage[leqno]{amsmath}
\usepackage{thmtools,mathtools}
\usepackage{subcaption,thm-restate}
\usepackage{tcolorbox}

\DeclareMathOperator\tss{Min-TSS}

\DeclareMathOperator\rank{rank}

\newcommand{\disttonh}[1]{\operatorname{dist^{nh}_\mathnormal{#1}}}
\newcommand{\disttorec}[1]{\operatorname{dist^r_\mathnormal{#1}}}

\DeclareMathOperator\ts{ts}

\newtcolorbox{probbox}{arc=6pt,
                      colback=white!100,
                      colframe=black!50,
                      before skip=6pt,
                      after skip=6pt,
                      boxsep=1pt,
                      left=6pt,
                      right=6pt,
                      top=4pt,
                      bottom=4pt}

\newcommand{\searchprob}[3]{
   \begin{center}%
    \begin{minipage}{\linewidth}%
      \begin{probbox}
      \textsc{#1}\\[0.2ex]
      \textbf{Input:} #2\\[0.2ex]
      \textbf{Goal:} #3
      \end{probbox}
    \end{minipage}%
  \end{center}
}

\theoremstyle{plain}
\newtheorem{thm}{Theorem}

\newtheorem{cl}[thm]{Claim}

\newtheorem{prop}[thm]{Proposition}

\theoremstyle{definition}
\newtheorem{rem}[thm]{Remark}

\title{On approximating the rank of graph divisors}

\author{
Kristóf Bérczi\thanks{MTA-ELTE Momentum Matroid Optimization Research Group and MTA-ELTE Egerváry Research Group, Department of Operations Research, Eötvös Loránd University, Budapest, Hungary. Email: \texttt{kristof.berczi@ttk.elte.hu}.}
\and
Hung P. Hoang\thanks{Institute of Theoretical Computer Science, Department of Computer Science, ETH Zürich, Switzerland. Email: \texttt{hung.hoang@inf.ethz.ch}}
\and
Lilla Tóthmérész\thanks{MTA-ELTE Egerváry Research Group, Department of Operations Research, Eötvös Loránd University, Budapest, Hungary. Email: \texttt{lilla.tothmeresz@ttk.elte.hu}.}
}
\date{}

\begin{document}

\maketitle

\begin{abstract}
Baker and Norine initiated the study of graph divisors as a graph-theoretic analogue of the Riemann-Roch theory for Riemann surfaces. One of the key concepts of graph divisor theory is the {\it rank} of a divisor on a graph. The importance of the rank is well illustrated by Baker's {\it Specialization lemma}, stating that the dimension of a linear system can only go up under specialization from curves to graphs, leading to a fruitful interaction between divisors on graphs and curves. 

Due to its decisive role, determining the rank is a central problem in graph divisor theory. Kiss and T\'othm\'er\'esz reformulated the problem using chip-firing games, and showed that computing the rank of a divisor on a graph is NP-hard via reduction from the Minimum Feedback Arc Set problem. 

In this paper, we strengthen their result by establishing a connection between chip-firing games and the Minimum Target Set Selection problem. As a corollary, we show that the rank is difficult to approximate to within a factor of $O(2^{\log^{1-\varepsilon}n})$ for any $\varepsilon > 0$ unless $P=NP$. Furthermore, assuming the Planted Dense Subgraph Conjecture, the rank is difficult to approximate to within a factor of $O(n^{1/4-\varepsilon})$ for any $\varepsilon>0$.

\medskip

\end{abstract}


\section{Introduction}
\label{sec:intro}

Motivated by the fact that a finite graph can be viewed as a discrete analogue of a Riemann surface, Baker and Norine~\cite{baker2007riemann} introduced a discrete analogue of the Riemann-Roch theory for graphs. They adapted the notions of a divisor, linear equivalence, and rank to the combinatorial setting, and studied the analogy between the continuous and the discrete cases in the context of linear equivalence. They further provided several results showing that the notion of rank plays an important role in various problems. In particular, a graph-theoretic analogue of the Riemann-Roch theorem holds for this notion of rank. Therefore, deciding the complexity of computing the rank is a natural question that is of both algebraic geometric and combinatorial interest. This problem is attributed to H. Lenstra, and was explicitly posed by many, see~\cite{baker2013chip,hladky2013rank,manjunath2011rank}. 

\paragraph{Previous work}

From the positive side, Luo~\cite{luo2011rank} introduced the notion of rank-determining sets of metric graphs, and verified the existence of finite rank-determining sets constructively. Hladk\'y, Kr\'al, and Norine \cite{hladky2013rank} confirmed a conjecture of Baker~\cite{baker2008specialization} relating the ranks of a divisor on a graph and on a tropical curve, and provided a purely combinatorial algorithm for computing the rank of a divisor on a tropical curve, which can be considered simply as a metric graph, see~\cite{mikhalkin2008tropical,gathmann2008riemann}. For multigraphs with a constant number of vertices, Manjunath~\cite{manjunath2011rank} gave a polynomial time algorithm that computes the rank of a divisor. Furthermore, by a corollary of the Riemann-Roch theorem, the rank can be computed in polynomial time for divisors of degree greater than $2g-2$, where $g$ denotes the cyclotomic number of the graph (also called the genus). Baker and Shokrieh~\cite{baker2013chip} provided an algorithm that can efficiently check whether the rank of a divisor on a graph is at least $c$ for any fixed constant $c$. Cori and le Borgne~\cite{cori2013riemann} provided a linear time algorithm for determining the rank of a divisor on a complete graph. 

There are also negative results regarding the complexity of computing the rank. Amini and Manjunath~\cite{amini2010riemann} proposed a general Riemann-Roch theory for sub-lattices of the root lattice $A_n$\footnote{$A_n$ is the lattice of $(n+1)$-dimensional integer vectors whose components sum to zero.}, thus giving a geometric interpretation of the work of Baker and Norine and generalising it to sub-lattices of $A_n$. They also showed that in the general model deciding whether the rank of a divisor is at least zero is already NP-hard. The complexity of computing the rank of a divisor on a graph was settled by Kiss and T\'othm\'er\'esz~\cite{divisor_rank_NP_hard} who showed that the problem is NP-hard even in simple graphs. 

\paragraph{Our results}

As computing the rank of a divisor is NP-hard in general, it is natural to ask whether it can be approximated within any reasonable factor. Our main contribution is establishing a connection between computing the rank and finding a so-called minimum target set in an undirected graph, a central problem in combinatorial optimization that is notoriously hard to approximate. A recent paper by B\'erczi, Boros, {\v{C}}epek, Ku{\v{c}}era, and Makino~\cite{berczi2022unique} showed a similar connection between finding a minimum key in a directed hypergraph and finding a minimum target set in an undirected graph, therefore, roughly speaking, our result places these three classical problems on the same level of complexity. As a corollary, we deduce strong lower bounds on the approximability of the rank. In particular, we show that the rank is difficult to approximate to within a factor of $O(2^{\log^{1-\varepsilon}n})$ for any $\varepsilon > 0$ unless $P=NP$. Furthermore, assuming the Planted Dense Subgraph Conjecture, we show that the rank is difficult to approximate to within a factor of $O(n^{1/4-\varepsilon})$ for any $\varepsilon>0$, thus providing the first polynomial hardness result for the problem. By \cite[Theorem 1.6]{luo2011rank}, our results also apply to computing the rank of a divisor on a tropical curve.
\medskip

The rest of the paper is organized as follows. Basic definitions and notation are introduced in Section~\ref{sec:prelim}, together with a brief introduction into graph divisor theory. The hardness of approximation of the rank of divisors is then discussed in Section~\ref{sec:main}.

\section{Preliminaries}
\label{sec:prelim}

\paragraph{Basic notation}

We denote the sets of \emph{integer} and \emph{nonnegative integer} numbers by $\mathbb{Z}$, and $\mathbb{Z}_+$, respectively. Let $V$ be a ground set and $X\subseteq V$ be a subset of $V$. 
For a function $f:V\to\mathbb{Z}$, we use $f(X)\coloneqq\sum_{v\in X}f(v)$. 

Let $G=(V,E)$ be a graph with \emph{vertex set} $V$ and \emph{edge set} $E$. Throughout the paper, we allow multiple edges between a pair of vertices. Given  $Z\subseteq V$ and $F\subseteq E$, the \emph{degree of $Z$ in $F$} is the number of edges in $F$ going between $Z$ and $V\setminus Z$ and is denoted by $d_F(Z)$. In particular, for a vertex $v\in V$, we denote by $d_F(v)$ the degree of $v$ in $F$. The \emph{degree vector} or \emph{degree sequence} of $G$ is then the vector $d_G\in\mathbb{Z}^V_+$ with $d_G(v)\coloneqq d_E(v)$. For two vertices $u,v\in V$, $d_G(v,u)$ refers to the \emph{number of edges connecting $v$ and~$u$} in~$G$.

\paragraph{Chip-firing}
We give a brief overview of the basic definitions of graph divisor theory due to Baker and Norine~\cite{baker2007riemann}. For a graph $G=(V,E)$, a \emph{divisor} is an assignment $f:V\to \mathbb{Z}$ of an integer number to each vertex. We usually refer to the value $f(v)$ as the \emph{number of chips} on $v$, even though $f(v)$ might be negative, and hence we also refer to $f$ as a \emph{chip configuration}. The \emph{degree} of a divisor $f$ is defined as $\deg(f)\coloneqq f(V)=\sum_{v\in V} f(v)$. The set of divisors admits an equivalence-relation, called \emph{linear equivalence}, which is defined based on the operation of \emph{firing a vertex}. For an arbitrary vertex $v$, the firing of $v$ modifies the divisor $f$ to $f'$ with 
\[f'(u)=\begin{cases} f(v)-d_G(v) & \text{if } u=v, \\
f(u)+d_G(v, u) & \text{if } u\neq v.
\end{cases}\]
That is, $v$ sends a chip along each edge incident to it, and as a result, the number of chips on $v$ decreases by the degree of $v$. Note that the degree of the divisor is not changed by the firing. 

Two divisors $f$ and $g$ are called \emph{linearly equivalent} if $f$ can be transformed into $g$ by a sequence of firings, denoted by $f\sim g$. Note that linear equivalence is indeed an equivalence relation, since a firing of $v$ can be `reversed' by firing every other vertex once. A divisor $f$ is called \emph{effective} if $f(v)\geq 0$ for each $v\in V$. A divisor is called \emph{winnable} if it is linearly equivalent to an effective divisor. A basic quantity associated to a divisor is its \emph{rank}, defined as \[\rank_G(f) \coloneqq \min\{\deg(g) \mid \text{ $g:V\to\mathbb{Z}$ is effective, $f-g$ is not winnable}\}-1.\]
In other words, because of the minus one in the definition, the rank of a divisor measures the maximum number of chips that can be removed arbitrarily from $f$ so that the resulting divisor is winnable.

We say that a vertex $v$ is \emph{active} with respect to a divisor $f$ if $f(v)\geq d_G(v)$. The \emph{chip-firing game} is a one player game, where a step consists of choosing an active vertex, if there exists any, and firing it\footnote{In some works, non-active vertices are also allowed to fire, resulting in a negative number of chips on them. In such a setting, a firing of a vertex $v$ is called \emph{legal} if $v$ is active, and games consisting of legal firings are called \emph{legal} as well. As we only consider legal games in this work, we dismiss the word `legal' throughout.}. The game is \emph{non-trivial} if at least one vertex gets fired. We say that the game \emph{halts} if we get to a divisor that has no active vertex; such a divisor is called \emph{stable}. The milestone result of Bj{\"o}rner, Lov{\'a}sz, and Shor~\cite{bjorner1991chip} states that starting from a given divisor, either every game halts and ends with the same stable divisor, or every game continues indefinitely. That is, the choices of the player do not influence whether the game halts or not, as it only depends on the starting divisor. Based on this theorem, we call a divisor $f$ \emph{halting} if the chip-firing game starting from $f$ eventually halts, and \emph{non-halting} otherwise. 

Recall that $d_G$ denotes the divisor that for any vertex $v$ has $d_G(v)$ chips. Let us denote by $\mathbf{1}$ the divisor that has 1 chip on every vertex. The winability of a divisor can be characterized by the chip-firing game the following way, see \cite[Corollary 5.4]{baker2007riemann}.

\begin{prop}[Baker and Norine]\label{prop:baker}
A divisor $f$ on a graph $G$ is winnable if and only if $d_G - \mathbf{1} - f$ is a halting divisor.
\end{prop}

\begin{rem}
In \cite{baker2007riemann} it is assumed that $d_G - \mathbf{1} - f$ has a nonnegative amount of chips on each vertex. However, the theorem remains true without this assumption (see for example \cite[Remark 2.6]{divisor_rank_NP_hard}), hence we do not suppose nonnegativity of divisors in the chip-firing game.
\end{rem}

To measure the deviation of a divisor $f$ of $G$ from being non-halting, we introduce its \emph{distance from a non-halting state} \[\disttonh{G}(f) \coloneqq \min\{\deg(g) \mid  g \text{ is an effective divisor of $G$}, f+g \text{ is non-halting}\}.\] Using this notation, the rank of a divisor can be formulated as  $\rank_G(f) = \disttonh{G}(d_G - \mathbf{1} - f) - 1$. Hence computing the rank of a divisor is equivalent to computing the distance of a related divisor, that is, the minimum number of chips that needs to be added to the related divisor to make the chip-firing game non-halting.

\searchprob{Dist-Nonhalt}{A graph $G=(V,E)$ and a divisor $f:V\to\mathbb{Z}$.}{Compute the distance $\disttonh{G}(f)$ of $f$ from a non-halting state.}

The following result provides a sufficient condition for a divisor $f$ being non-halting, see \cite[Lemma 4]{tardos88}.

\begin{prop}[G. Tardos] \label{prop:everyone_fired_implies_non-halting}
Let $G=(V,E)$ be a graph and $f:V\to\mathbb{Z}$ be a divisor. If there is a chip-firing game starting from $f$ where each vertex fires at least once, then $f$ is a non-halting configuration.
\end{prop}

A divisor $f$ is called \emph{recurrent} if there is a non-trivial chip-firing game starting from $f$ that leads back to $f$. Clearly, a recurrent divisor $f$ is also non-halting. Besides the distance of a divisor from a non-halting state, its \emph{distance from a recurrent state} also plays a central role in our investigations, defined as \[\disttorec{G}(f)\coloneqq \min\{\deg(g)\mid g \text{ is an effective divisor of $G$, $f+g$ is recurrent}\}.\]

\searchprob{Dist-Rec}{A graph $G=(V,E)$ and a divisor $f:V\to\mathbb{Z}$.}{Compute the distance $\disttorec{G}(f)$ of $f$ from a recurrent state.}

We will use the following characterization of recurrence, due to Biggs \cite[Lemma 3.6]{Biggs}. As~\cite{Biggs} uses a slightly different model, we include a short proof here.

\begin{prop}\cite[Lemma 3.6]{Biggs}\label{prop:recurrence_char}
Let $G=(V,E)$ be a connected graph and $f:V\to\mathbb{Z}$ be a divisor. Then $f$ is recurrent if and only if there is a game starting from $f$ in which each vertex fires exactly once.  
\end{prop}
\begin{proof}
The `if' direction is obvious, hence we concentrate on the `only if' direction. Assume that $f$ is a recurrent divisor. Then, by definition, there is a non-trivial game that transforms $f$ to itself. We claim that in such a game each vertex fires the same number of times. Indeed, otherwise let $Z\subsetneq V$ denote the proper subset of vertices that fired a minimum number of times. As every vertex in $V\setminus Z$ fired more times than the vertices in $Z$ and $G$ is connected, $Z$ necessarily gains chips compared to $f$, a contradiction.  By \cite[Lemma 4.3]{BL92}, there also exists a game starting from $f$ in which each vertex fires exactly once, concluding the proof.
\end{proof}

\paragraph{Minimum Target Set Selection}

In the Minimum Target Set Selection problem ({\sc Min-TSS}), introduced by Kempe, Kleinberg, and \'E.~Tardos~\cite{kempe2015maximizing}, we are given a simple graph $G=(V,E)$ with $|V|=n$, together with a threshold function $\tau:V\to\mathbb{Z}_+$, and we consider the following activation process. As a starting step, we can activate a subset $S\subseteq V$ of vertices. Then, in every subsequent round, a vertex $v$ becomes activated if at least $\tau(v)$ of its neighbors are already active. Note that once a vertex is active, it remains active in all future rounds. The goal is to find a minimum sized initial set $S$ of active vertices (called a \emph{target set}) so that the activation spreads to the entire graph (i.e., all vertices are eventually activated). For ease of notation, we measure the size of an optimal solution by \[\ts_G(\tau)\coloneqq\min\{|S|\mid S\subseteq V,\ \text{$S$ is a target set of $G$ with respect to $\tau$}\}.\]

\searchprob{Min-TSS}{A simple graph $G=(V,E)$ and thresholds $\tau:V\to\mathbb{Z}_+$.}{Find a minimum sized target set attaining $\ts_G(\tau)$.}

{\sc Min-TSS} has long been a focus of research due to its wide applications in real life scenarios. Unfortunately, not only finding a smallest target set but also approximating the minimum value is difficult. 
Under the assumption $P\neq NP$, Feige and Kogan~\cite[Theorem 3.5]{feige2021target} showed that the problem is difficult even in 3-regular graphs.

\begin{prop}[Feige and Kogan]\label{prop:feige}
Unless $P=NP$, \textsl{\textsc{Min-TSS}} cannot be approximated to within a factor of $O(2^{\log^{1-\varepsilon} n})$ for any $\varepsilon>0$ even in $3$-regular graphs with $\tau(v)\in\{1,2\}$ for $v\in V$.
\end{prop}

Also, Charikar, Naamad, and Wirth~\cite{charikar2016approximating} proved that, assuming the Planted Dense Subgraph Conjecture, {\sc Min-TSS} is difficult to approximate within a polynomial factor.

\begin{prop}[Charikar, Naamad, and Wirth]\label{prop:charikar}
Assuming the Planted Dense Subgraph Conjecture, {\sc Min-TSS} cannot be approximated to within a factor of $O(n^{1/2-\varepsilon})$ by a probabilistic polynomial-time algorithm for any $\varepsilon>0$.  
\end{prop}

\section{Hardness of approximating the rank}
\label{sec:main}

Now we turn to the proof of the main result of the paper, and show that approximating the rank of a graph divisor within reasonable bounds is hard. The high level idea of the proof is the following. First, we show that the Minimum Target Set Selection problem reduces to computing the distance of a divisor on an auxiliary undirected graph from a recurrent state (Section~\ref{sec:red1}). Then we show that computing the distance of a divisor from a recurrent state reduces to computing the distance of a divisor on a slightly modified graph from a non-halting state (Section~\ref{sec:red2}). As the computation of the rank is equivalent to computing the distance from a non-halting state for an appropriately modified divisor, the lower bounds on the approximability of the rank will follow by Propositions~\ref{prop:feige} and \ref{prop:charikar}; see Theorem~\ref{thm:main}. 

\subsection{Reduction of {\sc Min-TSS} to {\sc Dist-Rec}}
\label{sec:red1}

Consider an instance $G=(V,E)$ and $\tau:V\to\mathbb{Z}_+$ of {\sc Min-TSS}. We construct a {\sc Dist-Rec} instance $G'=(V',E')$ and $x:V'\to\mathbb{Z}$ such that $\ts_G(\tau)=\disttorec{G'}(x)$.

We define $G'$ as follows. For each $v\in V$, we add three vertices $v_i$, $v_c$ and $v_o$ to $V'$. In addition, for each edge $e=uv\in E$, we add two vertices $p_{uv}$ and $p_{vu}$ to $V'$. Now let $N=|V|+2$. For each vertex $v\in V$, we add $N$ parallel edges between $v_i$ and $v_c$ and a single edge between $v_c$ and $v_o$. Furthermore, for each edge $e=uv\in E$, we add $N$ parallel edges between $u_o$ and $p_{uv}$, a single edge between $p_{uv}$ and $v_i$, $N$ parallel edges between $v_o$ and $p_{vu}$, and a single edge between $p_{vu}$ and $u_i$. These edges together form $E'$. For ease of discussion, we denote by $V'_\circ\coloneqq\{v_i,v_o\mid v\in V\}$ and $V'_\bullet\coloneqq\{v_c\mid v\in V\}\cup\{p_{uv},p_{vu}\mid uv\in E\}$.

\begin{figure}[t!]
\centering
\begin{subfigure}[t]{0.48\textwidth}
  \centering
  \includegraphics[width=.5\linewidth]{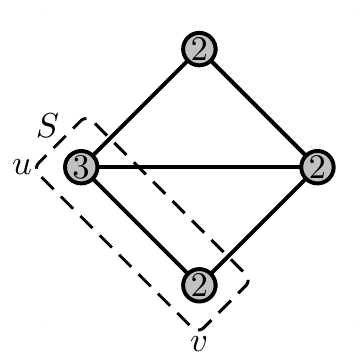}
  \caption{An instance of {\sc Min-TSS}. The numbers denote the thresholds at the vertices. The set $S=\{u,v\}$ is a target set; it is not difficult to check that $S$ has minimum size.}
  \label{fig:red_1}
\end{subfigure}\hfill
\begin{subfigure}[t]{0.48\textwidth}
  \centering
  \includegraphics[width=.5\linewidth]{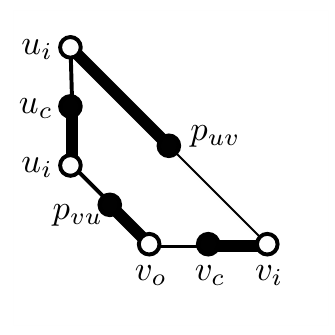}
  \caption{The subgraph of $G'$ induced by the vertices $\{v_i,v_o,u_i,u_o\}\subseteq V'_\circ$ and $\{v_c,u_c,p_{uv},p_{vu}\}\subseteq V'_\bullet$. Thick edges denote a bundle of $N=|V|+2=6$ parallel edges.}
  \label{fig:red_3}
\end{subfigure}
\vspace{10pt}

\begin{subfigure}[t]{\textwidth}
  \centering
  \includegraphics[width=.3\linewidth]{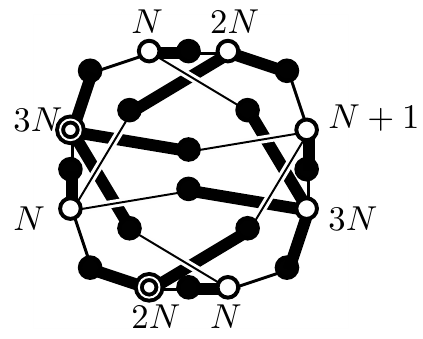}
  \caption{The corresponding {\sc Dist-Rec} instance. The numbers denote the initial chip distribution $x$ where $x(v)=1$ for each $v\in V'_\bullet$. Adding one chip on vertices with double border results in a recurrent divisor.}
  \label{fig:red_2}
\end{subfigure}
\caption{An illustration of the reduction of {\sc Min-TSS} to {\sc Dist-Rec}.}
\label{fig:red}
\end{figure}

We define the chip configuration $x$ in the following way. For each $v\in V$, let $x(v_o)\coloneqq d_{G'}(v_o)-1=N\cdot d_G(v)$ and $x(v_i)\coloneqq d_{G'}(v_i)-\tau(v)=d_G(v)+N-\tau(v)$. For each $z\in V'_\bullet$, let $x(z)\coloneqq d_{G'}(z)-N=1$. Note that $0\leq x(z)\leq d_{G'}(z)$ holds for each $z\in V'$. For an illustration of the construction, see Figure~\ref{fig:red}.

We claim that $\ts_G(\tau)=\disttorec{G'}(x)$. We prove the two directions separately.

\begin{cl}\label{cl:disstorec_tss_1}
$\ts_G(\tau)\geq\disttorec{G'}(x)$.
\end{cl}
\begin{proof}
Take a minimum sized target set $S\subseteq V$. Let $y\in \mathbb{Z}^{V'}$ be defined as 
\begin{align*}
    y(z)\coloneqq\begin{cases} 1 & \text{if $z=v_o$ for some $v\in S$}, \\
0 & \text{otherwise.} 
\end{cases}
\end{align*}
We claim that $x+y$ is a recurrent chip-configuration on $G'$. To see this, by Proposition~\ref{prop:recurrence_char}, it suffices to show that every vertex can fire exactly once starting from $x+y$. If $u\in S$, then $(x+y)(u_o)=d_{G'}(u_o)$, hence these vertices can fire without receiving any further chips from other vertices; fire all of them in an arbitrary order. If $u_o$ fires, then the vertex $p_{uv}$ receives $N$ chips for each edge $uv\in E$, therefore, $p_{uv}$ can fire. Fire them when we can. If a vertex $v\in V$ has at least $\tau(v)$ neighbors in $S$, then $v_i$ received at least $\tau(v)$ chips up to this point, that is, the total number of chips on $v_i$ is at least $d_{G'}(v_i)$. Thus, $v_i$ can fire as well. Firing $v_i$ sends $N$ chips to $v_c$, hence at this point $v_c$ can also fire, which results in sending one chip to $v_o$ and making a firing at $v_o$ possible.

Continuing this way, we get that if a vertex $v\in V$ becomes active in the activation process in the {\sc Min-TSS} instance, then $v_o$ can fire at a certain point. Hence eventually, there is a moment when $v_o$ has fired once for every $v\in V$. This means that for each edge $uw\in E$ and any of the vertices $p_{uw}$ and $p_{wu}$, either it fired once or has enough chips to do so; in the latter case, let it fire. After these, for each vertex $v\in V$, the vertex $v_i$ has either fired once or has enough chips to do so; in the latter case, let it fire, resulting in a chip configuration for which $v_c$ can fire as well. We conclude that eventually each vertex fires (exactly) once, hence $(x+y)$ is recurrent by Proposition~\ref{prop:recurrence_char} and $y$ is effective.

By the above, $\ts_G(\tau)=|S|=\deg(y)\geq \disttorec{G'}(x)$, concluding the proof of the claim.
\end{proof}

\begin{cl}\label{cl:disttorec_tss_2}
$N>\disttorec{G'}(x)+1$.
\end{cl}
\begin{proof}
Observe that $\ts_G(\tau)\leq |V|$ clearly holds. By Claim~\ref{cl:disstorec_tss_1}, we get $N=|V|+2\geq\ts_G(\tau)+2>\disttorec{G'}(x)+1$ as stated.
\end{proof}

\begin{cl}\label{cl:disttorec_tss_3}
$\ts_G(\tau)\leq \disttorec{G'}(x)$.
\end{cl}
\begin{proof}
Let $y$ be a chip configuration such that $y\geq 0$, $x+y$ is recurrent, and $\deg(y)=\disttorec{G'}(x)$. By Proposition~\ref{prop:recurrence_char}, there is a game starting from $x+y$ in which each vertex fires exactly once; from now on, we consider such a sequence of firings.

Consider a vertex $z\in V'_\bullet$ and let $w$ be the unique vertex in $V'_\circ$ that is connected to $z$ by $N$ parallel edges. By definition, we have $x(z)=1$. If $z$ fires before $w$ in the game, then $z$ can receive at most one chip from its other neighbor before firing. That is, $y(z)\geq N-1$ as otherwise $z$ has not enough chips for firing before $w$. By Claim~\ref{cl:disttorec_tss_2}, we have $\disttorec{G'}(x)< N-1$, a contradiction. Hence $z$ necessarily fires after $w$ in the game, meaning that it receives $N$ chips from $w$. That is, it has enough chips to fire even if $y(z)=0$. As $z$ only fires once and $\deg(y)$ is as small as possible, we conclude that $y(z)=0$. Moreover, $z$ fires later than $w$ for each $z\in V'_\bullet$, where $w$ is the unique neighbor of $z$ that is connected to $z$ by $N$ parallel edges. 

Now consider a vertex $v_i$ for some $v\in V$. We show that we may assume $y(v_i)=0$. Indeed, suppose that $y(v_i)=k>0$. Let $\{u^1, \dots, u^d\}$ be the neighbors of $v$ in $G$.
We create a new $y'$ by replacing the chips of $y$ from $v_i$ to some $u^j_o$ such that $y'$ remains effective with $\deg(y')\leq \deg(y)$, $x+y'$ is also recurrent, but $y'(v_i)=0$. (The minimality of $\deg(y)$, then of course implies $\deg(y)=\deg(y')$.)
Take a game started from $x+y$ in which each vertex of $G'$ fires exactly once. We know that $v_c$ fires after $v_i$ in this game. We have $(x+y)(v_i)=\deg_{G'}(v_i)-\tau(v) + k$. Hence by the time $v_i$ fires, at least $\max\{0,\tau(v) -k\}$ vertices among $\{p_{u^1v}, \dots, p_{u^dv}\}$ fired. Choose $k$ vertices $u_o^j$ such that $p_{u^jv}$ did not fire (or if there are less such vertices, choose all of them), and put $y'(u_o^j)=1$ for each of them. Moreover, put $y'(v_i)=0$. On all other vertices, let $y(w)=y'(w)$ This way, $y'\geq 0$ with $\deg(y')\leq\deg(y)$. Then, started from $x+y'$, we can play the same game as the one started from $x+y$ up to the firing of $v_i$. Moreover, if $y'(u^j_o)=1$, then $u_o^j$ and $p_{u^jv}$ can also fire as $x(u^j_o)=d_{G'}(u^j_o)-1$. Hence $v_i$ receives at least $\tau(v)$ chips. Thus, $v_i$ can also fire, and we can complete the chip-firing game. (We only brought forward the firings of $u_o^j$ and $p_{u^jv}$ for some $j$.)

Finally, consider a vertex $v_o$ for some $v\in V$. As each vertex fires exactly once, we get $y(v_o)\leq d_{G'}(v_o)-x(v_o)=d_{G'}(v_o)-(d_{G'}(v_o)-1)=1$. 

Define $S=\{v\in V\mid y(v_o)=1\}$. Our goal is to show that $S$ is a target set of the {\sc Min-TSS} instance $(G,\tau)$. We prove this by showing that the game on $G'$ starting from $x+y$ mimics the activation process on $G$ starting from $S$ in the following sense: by the time a vertex $v_o$ becomes active in the chip firing sense, its corresponding vertex $v\in V$ becomes active in the activation process on $G$. This clearly holds for vertices in $S$. Consider a general phase of the game, and assume that the condition holds, that is, for each vertex $v_o$ that was fired so far, $v$ became active in the {\sc Min-TSS} instance as well. Let $v\in V\setminus S$ and assume that $v_o$ is the next vertex to fire. Recall that $y(v_o)=0$ and that $v_o$ has to fire before the vertex $p_{vu}$ for every $uv\in E$. Apart from these vertices, $v_o$ has only one neighbor, namely $v_c$. This means that for $v_o$ to have enough chips, $v_c$ had to fire before $v_o$. As $v_i$ had to fire before $v_c$, we get that $v_i$ had to fire before $v_o$. As $y(v_i)=0$, $v_i$ needed $\tau(v)$ extra chips to be able to fire. As $v_c$ fired after $v_i$, these extra chips could only be procured from vertices $p_{uv}$ where $uv\in E$. As each of these vertices fired at most once and each of them sent one chip to $v_i$ upon firing, at least $\tau(v)$ of them had to fire before firing $v_i$. This means that at least $\tau(v)$ of the vertices $u_o$ with $uv\in E$ had to fire so far. By our hypothesis, these vertices $u\in V$ were active in the {\sc Min-TSS} instance by the time they fired, implying that $v$ is also active in the {\sc Min-TSS} instance.

By the above, $\ts_G(\tau)\leq |S|=\deg(y)=\disttorec{G'}(x)$, concluding the proof of the claim.
\end{proof}

Note that the size of $G'$ is polynomial in the size of $G$. In particular, we have $|V'|=3\cdot |V|+2\cdot |E|=O(|V|^2)$ as $G$ is simple.

\subsection{Reduction of {\sc Dist-Rec} to {\sc Dist-Nonhalt}}
\label{sec:red2}

Consider a {\sc Dist-Rec} instance $G=(V,E)$ and $f:V\to\mathbb{Z}$. We construct a {\sc Dist-Nonhalt} instance $G'=(V',E')$ and $f':V'\to\mathbb{Z}$ such that $\disttorec{G}(f)=\disttonh{G'}(f')$. Note the following.

\begin{cl}\label{cl:M}
$\disttorec{G}(f)\leq 2\cdot|E|+\displaystyle\sum_{v\in V: f(v)<0}|f(v)|.$
\end{cl}
\begin{proof}
If we add $\max\{0,d_G(v)-f(v)\}$ chips to vertex $v$, then we have a chip configuration with at least $d_G(v)$ chips on each vertex $v$, which is certainly a recurrent configuration. Moreover, this addition needs at most the above stated number of chips.
\end{proof}

We define $G'$ as follows. Let $M$ be an integer large enough such that \[\disttorec{G}(f)+\max\left\{0,\max_{v\in V}\{f(v)-d_G(v)\}\right\} < M.\] 
By Claim~\ref{cl:M}, $M\coloneqq2\cdot|E|+\sum_{v\in V: f(v)<0}|f(v)|+\max\left\{0,\max_{v\in V}\{f(v)-d_G(v)\}\right\}+1$ is sufficient. Let $V'\coloneqq V \cup \{v_{new}\}$ and let $E'$ consist of $E$ together with $M$ parallel edges between $v$ and $v_{new}$ for each vertex $v\in V$. That is, $|E'|=|E|+M\cdot |V|$. Finally, we define $f'(v)\coloneqq f(v)+M$ for $v\in V$ and $f'(v_{new})\coloneqq 0$.

We claim that $\disttorec{G}(f)=\disttonh{G'}(f')$. We prove the two directions separately.

\begin{cl}\label{cl:disttonh_disttorec_1}
$\disttorec{G}(f)\geq\disttonh{G'}(f')$.
\end{cl}
\begin{proof}
Let $g$ be an effective divisor of minimum degree such that there is a chip-firing game on $G$ starting from $f+g$ in which each vertex fires exactly once. We extend $g$ to $V'$ by setting $g(v_{new})=0$. We claim that if $\{v_1, \dots v_k\}$ is a sequence of firings starting from $f+g$ on $G$ in which each vertex occurs at most once, then these vertices can also fire in the same order starting from $f'+g$ on $G'$. Indeed, for each vertex $v\in V$, we have $d_{G'}(v)=d_G(v)+M$ and $f'(v)+g(v)=f(v)+g(v)+M$. Firing a vertex $v\in V$ modifies the chip-configuration on other vertices in $V$ the same way in $G$ and in $G'$. Hence if $v_i$ has not yet been fired and it has enough chips in the game on $G$, then it also has enough chips in the game on $G'$. This implies that there is a game on $G'$ starting from $f'+g$ in which each vertex of $V$ fires exactly once. At this point, $v_{new}$ received $M$ chips from each vertex in $V$, hence it has $M\cdot |V|$ chips. As its degree is also $M\cdot |V|$, at this point $v_{new}$ can fire, resulting in a game on $G'$ in which each vertex has fired exactly once.

By Proposition~\ref{prop:everyone_fired_implies_non-halting}, the above implies that $f'+g$ is non-halting on $G'$. Hence we have $\disttorec{G}(f)=\deg(g)\geq\disttonh{G'}(f')$, concluding the proof of the claim.
\end{proof}

\begin{cl}\label{cl:disttonh_disttorec_2}
$M>\disttonh{G'}(f')$.
\end{cl}
\begin{proof}
We choose $M$ such that $M>\disttorec{G}(f)$, hence the claim follows by Claim \ref{cl:disttonh_disttorec_1}.
\end{proof}

\begin{cl}\label{cl:disttonh_disttorec_3}
$\disttorec{G}(f)\leq\disttonh{G'}(f')$.
\end{cl} 
\begin{proof}
Let $g'$ be an effective divisor such that $f'+g'$ is non-halting on $G'$ and $\deg(g')=\disttonh{G'}(f')$. We show that $g'(v_{new})=0$, and that for the divisor $g\coloneqq g'|_{V}$ obtained by restricting $g'$ to the set of original vertices, the divisor $f+g$ is recurrent on $G$. 

Observe that $f'(v_{new})=0$ and $v_{new}$ receives $M$ chips each time a vertex in $V$ fires. Moreover, $d_{G'}(v_{new})$ is divisible by $M$, hence adding less than $M$ chips on $v_{new}$ has no effect on whether it is able to fire or not. As $\disttonh{G'}(f')<M$, we conclude that $g'(v_{new})=0$.

Now take any game on $G'$ starting from $f'+g'$. We claim that no vertex $v\in V$ can fire twice before the first firing of $v_{new}$. Suppose to the contrary that some vertex fires twice before the first firing of $v_{new}$, and let $v\in V$ be the first such vertex. Then before its second firing, $v$ received at most $d_G(v)$ chips from its neighbors, and it lost $d_{G'}(v)$ chips upon its first firing. Since, by our assumption, $v$ is now capable of firing for the second time, we get that $f'(v)+g'(v)+d_G(v)-d_{G'}(v) \geq d_{G'}(v)$. As $d_{G'}(v)=d_G(v)+M$, this implies $f'(v)+g'(v) \geq d_G(v) + 2M$. As $f'(v)=f(v)+M$, the choice of $M$ implies $g'(v)\geq d_G(v)-f(v) + M> \disttorec{G'}(f')\geq\disttonh{G'}(f')=\deg(g')$, a contradiction.

Hence each vertex $v\in V$ fires at most once before the first firing of $v_{new}$. As $v_{new}$ has to collect $M\cdot |V|$ extra chips to be able to fire, each vertex $v\in V$ fires exactly once before the first firing of $v_{new}$. These firings would also be possible when starting from $f+g$ on $G$, hence $f+g$ is recurrent by Proposition \ref{prop:recurrence_char}.
\end{proof}

Note that the size of $G'$ might not be polynomial in the size of $G$ as the value of $M$ can be large. However, the above reduction will be applied to divisors $f$ with $0\leq f(v)\leq d_G(v)$ for $v\in V$, in which case the construction is polynomial. In particular, observe that $|V'|=|V|+1$.

\subsection{Lower bounds}

By combining the observations of Sections~\ref{sec:red1} and \ref{sec:red2} with Propositions~\ref{prop:feige} and \ref{prop:charikar}, we are ready to prove the main result of the paper.

\begin{thm}
\label{thm:main}
Unless $P=NP$, {\sc Dist-Nonhalt} cannot be approximated to within a factor of $O(2^{\log^{1-\varepsilon}n})$ for any $\varepsilon>0$. Assuming the Planted Dense Subgraph conjecture, {\sc Dist-Nonhalt} cannot be approximated to within a factor of $O(n^{1/4-\varepsilon})$ for any $\varepsilon>0$.
\end{thm}
\begin{proof}
Let $G=(V,E)$ be a graph and $\tau:V\to\mathbb{Z}_+$ be a threshold function. According to Section~\ref{sec:red1}, computing the value $\ts_G(\tau)$ of the $\tss$ instance on $G$ and $\tau$ can be reduced to computing $\disttorec{G'}(x)$ for a graph $G'=(V',E')$ having $|V'|=3\cdot |V|+2\cdot |E|$ vertices and $|E'|=2\cdot (|V|+3)\cdot |E|+|V|$ edges, and a divisor $x:V'\to\mathbb{Z}_+$ with $0\leq x(v)\leq d_{G'}(v)$ for each $v\in V'$.

As $\ts_G(\tau)\leq |V|$, we have $\disttorec{G'}(x)\leq |V|$. Hence we can use the reduction of Section \ref{sec:red2} with $M=|V|+1$. Thus, we have a graph $G''=(V'',E'')$ with $|V''|=3\cdot|V|+2\cdot|E|+1$ vertices and $|E''|=2\cdot (|V|+3)\cdot |E|+2\cdot |V|+|V|^2$ edges, and a divisor $x'':V''\to\mathbb{Z}_+$ such that $\ts_G(\tau) = \disttonh{G''}(x'')$.

Combining the above with Propositions \ref{prop:feige} and \ref{prop:charikar}, the statements of the theorem follow.
\end{proof}

\begin{rem}
The above reduction of {\sc Min-TSS} to {\sc Dist-Nonhalt} uses an auxiliary graph that contains parallel edges. Therefore, the stated complexity results hold for general (i.e. not necessarily simple) graphs. However, in \cite[Corollary 22]{hladky2013rank}, Hladk\'{y}, Kr\'{a}\v{l}, and Norine proved that if one subdivides each edge of a graph with a new vertex and places 0 chips on the new vertices, then the rank of the obtained divisor on the new graph is the same as the rank of the orignal divisor on the original graph. Using this result, our construction can be modified in such a way that eventually we get a {\sc Dist-Nonhalt} instance on a simple graph, at the expense of having a higher number of vertices. Such a reduction, assuming the Planted Dense Subgraph Conjecture, still implies a polynomial hardness to the problem.
\end{rem}

\begin{rem}
The proofs of the paper are based on a construction that resulted in a graph with large degrees. Deciding the complexity of computing the rank in degree-bounded graphs remains an interesting open problem.
\end{rem}
\medskip

\paragraph{Acknowledgement} 
Part of the work was done while the first and second authors attended the ``Discrete Optimization'' trimester program of the the Hausdorff Institute of Mathematics, Bonn. The authors are grateful to HIM for providing excellent working conditions and support. 

The third author was supported by the János Bolyai Research Scholarship of the Hungarian Academy of Sciences, and by the ÚNKP-21-5 New National Excellence Program of the Ministry of Innovation and Technology of Hungary. The research has been implemented with the support provided by the Lend\"ulet Programme of the Hungarian Academy of Sciences -- grant number LP2021-1/2021, by the Hungarian National Research, Development and Innovation Office -- NKFIH, grant numbers FK128673, PD132488, and by the Ministry of Innovation and Technology of Hungary from the National Research, Development and Innovation Fund, financed under the ELTE TKP 2021-NKTA-62 funding scheme.

\bibliographystyle{abbrv}
\bibliography{rank}

\end{document}